\documentclass[1p]{elsarticle}
 
 \usepackage[latin1]{inputenc}

\journal{arXiv}

\usepackage{amssymb}
\usepackage{eucal}
\usepackage{graphicx}
\usepackage{tikz}
\usepackage{subcaption}
\usetikzlibrary{fadings}
\usetikzlibrary{patterns}
\usetikzlibrary{shadows.blur}
\usetikzlibrary{shapes}
\usepackage[normalem]{ulem}
\newcommand{\R}{{\Bbb R}}

\newcommand{\N}{{\Bbb N}}

\usepackage{color}

\usepackage{xcolor}
\definecolor{alizarin}{rgb}{0.82, 0.1, 0.26}
\definecolor{burgundy}{rgb}{0.5, 0.0, 0.13}
\definecolor{darkblue}{rgb}{0.0, 0.0, 0.55}
\definecolor{deepcarmine}{rgb}{0.66, 0.13, 0.24}
\definecolor{navyblue}{rgb}{0.0, 0.0, 0.5}

\usepackage{amsmath}

\newtheorem{thm}{Theorem}

\newtheorem{cor}[thm]{Corollary}

\newtheorem{proposition}[thm]{Proposition}

\newproof{proof}{Proof}

\begin{document}
\begin{frontmatter}

\title{Global dynamics of a  size-structured forest model}

\author[a]{Franco Herrera}
\author[a]{Sergei Trofimchuk\corref{mycorrespondingauthor}}
\cortext[mycorrespondingauthor]{\hspace{-7mm} {\it e-mails addresses}:   franco.herrera@utalca.cl (F. Herrera); trofimch@inst-mat.utalca.cl (S. Trofimchuk, corresponding author)  \\}
\address[a]{Instituto de Matem\'aticas, Universidad de Talca, Casilla 747, Talca, Chile}

\bigskip

\begin{abstract}
\noindent  We study  a size-structured model proposed recently by C. Barril et al to describe the dynamics of trees growth in the forest. Our approach to the  associated renewal equation is rather different from the methods  in   
the cited work and is based on ideas developed in earlier authors' work on the relation between dynamics of one-dimensional maps and Gurtin-MacCamy's population model.  Assuming relatively weak restrictions on the reproduction, death and growth  rates $\beta, \mu, g$,  we establish the permanence properties of the semiflow $\frak F^t$  generated by the renewal equation and prove that it possesses a  compact global attractor of points $\mathcal A$. Next we show that the opposite types of monotonicity  of $\beta, g$ assure that $\frak F^t$ is also monotone and that in this case $\mathcal A$ coincides with  a unique  asymptotically stable equilibrium attracting neighbourhoods of compact sets with non-zero initial data. In particular, by establishing good permanence properties of the semiflow $\frak F^t$, 
our results additionally supports  well-posedness of the model elaborated by C. Barril et al. 

 \end{abstract}

\end{frontmatter}

\newpage

\section{Introduction and main results}  \label{intro} 
A temperate  forest can be provided with a collective order based on the size (height) of trees. In this order,  due to the better conditions for the photosynthesis, the taller trees have higher hierarchy in comparison with the smaller ones. This consideration allows to model the forest dynamics within frameworks of the theory of physiologically structured populations, see \cite{MZ,MD,HS,GW} where further references can be found. One of such models was recently proposed by C. Barril et al  \cite{FBCD} in form of the renewal equation 
\begin{equation}\label{RE}
b(t) =\int_0^\infty \beta\left( x_m + \int_0^a g\left( e^{-\mu(\tau -a)} \int_a^\infty e^{-\mu s} b(t-s)\,ds \right) d\tau \right) e^{-\mu a} b(t-a)da,
\end{equation}
where $b(t)$ is the tree's population birth rate at time $t$, $g(x)$ denotes the growth rate  of an individual of height $x$,  the positive parameter $\mu >0$ represents the per capita death rate  and $\beta(x)$ is the per capita reproduction rate (depending only on height $x$). It is also assumed  in \cite{FBCD}  that all newborn individuals have the minimal height $x_m \geq 0$ and that 

\vspace{2mm}

\noindent {\bf (M)} Mappings $\beta:[x_m,+\infty)\to [0, +\infty)$,  $g: [0,+\infty)\to (0, +\infty)$ are continuous, $\beta$  is  increasing and $g$ is decreasing one with $g(+\infty)=0$. Moreover, the monotone function 
$$R(b):=\int_0^\infty \beta\left( x_m + \int_0^a g\left( \frac{e^{-\mu\tau}}{\mu}b\right) d\tau \right) e^{-\mu a}da$$  
is such that the pre-image $R^{-1}(1)$   is either a single point $b_*$ or the empty set.

\vspace{2mm}

Model (\ref{RE}) is considered as an integral equation with infinite memory and is provided with a non-negative initial condition 
$b(s)=\phi(s) \geq 0, \ s \leq 0$. Even if $\phi$ is a continuous function, $b(t)$ can have a jump discontinuity at $t=0$. In part, this  explains the choice of the Banach space of measurable functions (instead of spaces of continuous functions of fading memory type \cite{AH})
$$
L^1_\rho(\R_{-})= \{\phi: |\phi|_\rho < \infty\}, \quad \mbox{where we use the notation} \quad |\phi|_q:= \int_{-\infty}^0|\phi(s)|e^{qs}ds,
$$
with $\rho < \mu$ as an appropriate phase space for equation (\ref{RE}). Fixing an initial function $\phi \geq 0$, we can rewrite equation (\ref{RE}) in the Volterra's form $b(t) =  (\mathcal{V}_\phi b)(t) + r_\phi(t)$, where 
$$\hspace{-2mm}(\mathcal{V}_\phi b)(t) = \int_0^t \beta\left( x_m + \int_0^a g\left( e^{-\mu(\tau -a)}(\int_a^t e^{-\mu s} b(t-s)ds+ e^{-\mu t}|\phi|_\mu)\right)d\tau\right) e^{-\mu a}b(t-a)da, $$
$$r_\phi(t)= \int_t^\infty \beta\left( x_m + \int_0^a g\left( e^{-\mu(\tau -a)} \int_a^\infty e^{-\mu s} \phi(t-s)ds\right) d\tau \right) e^{-\mu a} \phi(t-a)da.
$$
Then a standard argumentation (outlined in Appendix) yields the following  existence and uniqueness result.  
\begin{proposition}\label{P1}
Assume that continuous function $\beta:[x_m,+\infty)\to [0, +\infty)$  has at most polynomial growth and  $g: [0,+\infty)\to [0, +\infty)$ is continuous and bounded. Then 
\begin{itemize}
\item[(a)] there exists at least one non-negative  function $b_\phi(t)$ which is continuous on 
$[0, +\infty)$ and solves the initial value problem $b(s)=\phi(s) \geq 0, \ s \leq 0$ for  (\ref{RE});
\item[(b)] in addition, if $\beta$ is locally Lipschitizian function then the solution $b_\phi$ is a unique one, it is continuously differentiable on $(0,+\infty)$ and the derivative $b_\phi'(t)$ is uniformly bounded on each interval $(0,T]$. 
\item[(c)] Moreover, if the assumption of item (b) is satisfied and $\phi \to \phi_0$ in $L^1_\rho(\R_{-})$ then $b_\phi(t) \to b_{\phi_0}(t)$ uniformly on compact subsets of 
$[0,+\infty)$. 
\end{itemize}
\end{proposition}

By relaxing assumptions of growth and continuity on $\beta$ and $g$,  Proposition \ref{P1} improves \cite[Theorem 3.2]{FBCD}  and shows that if $\beta$ is locally Lipschitizian then 
the renewal equation  (\ref{RE}) generates a continuous semiflow 
\mbox{$\frak F: L^1_\rho(\R_{-}) \times \R_+ \to L^1_\rho(\R_{-})$} by  $(\frak F^t \phi)(s) = b(t+s),$ $s \leq 0$. It has a zero 
steady state $b(t) \equiv 0$ while  the positive equilibria of this semiflow are defined from the scalar equation 
 \begin{equation}\label{ERE}
1= R(b)=\int_0^\infty \beta\left( x_m + \int_0^a g\left( \frac{e^{-\mu\tau}}{\mu}b\right) d\tau \right) e^{-\mu a}da. 
\end{equation}
Note that under hypothesis {\bf (M)} function  $R(b)$ is decreasing on $\R_+$ and the latter equation has (a unique positive) solution $b_*$ if and only if 
$R(0) >1$ and $R(+\infty)= \beta(x_m)/\mu <1$, cf. \cite[Theorem 4.1]{FBCD}. A straightforward application of the comparison argument shows that all solutions of equation (\ref{RE}) are exponentially converging to $0$ when $R(0) <1$, 
see Theorem 5.2 in \cite{FBCD} (as we show it later,  under hypothesis {\bf (M)} all solutions vanish at $+\infty$ even if $R(0)=1$). Thus $R(0)$ can be interpreted as the basic reproduction number. 

We will be interested in eventually positive solutions of equation (\ref{RE}) and the following dichotomy result will be 
sufficient for our purposes. In its statement, we consider the operator ${\mathcal F}: L^1_\rho(\R_{-}, \R_+) \to \R_+$ defined by
\begin{equation}\label{OF}
{\mathcal F}\phi:=\int_0^\infty \beta\left( x_m + \int_0^a g\left( e^{-\mu(\tau -a)} \int_a^\infty e^{-\mu s} \phi(-s)\,ds \right) d\tau \right) e^{-\mu a} \phi(-a)da. 
\end{equation}

\begin{proposition}\label{P2}
Assume that $\beta(x) > 0$ a.e. on $[x_0,+\infty)$  for some $x_0\geq x_m$  and either $\inf_{x\geq 0} g(x) >0$ or 
$g(x)>0$ and $b(t)$ is a solution bounded  on $\R_+$.  Then either $b(t) =0$ for all $t \geq 0$ (so that ${\mathcal F}\phi =0$)  or  there exists some $t_0 \geq 0$ such that $b(t) >0$ for all $t \geq t_0$. 
\end{proposition}

In the sequel, we will use the notation $\mathcal N$ for the non-empty closed subset of all initial data $\phi,\ {\mathcal F}\phi =0$, which generate the zero continuation $b(t)\equiv 0$ on $\R_+$.

Assume now that equation  (\ref{RE})  has a positive equilibrium $b_*$ (i.e. $R(0) >1$, $\beta(x_m) < \mu$). 
One of basic questions concerns its stability properties (with respect to semiflow $\frak F^t \phi$). Using 
the principle of linearised stability for delay equations from \cite{DG}, the authors of \cite{FBCD} proved the local asymptotical   stability of the positive equilibrium $b_*$ for the particular case when $x_m=0$ and $\beta(x)= \max\{0, x- x_A\}$
with some $x_A>0$ (see Theorem 5.4 in \cite{FBCD}). In particular, their approach required certain technical efforts while 
establishing the continuous differentiability of the operator ${\mathcal F}: L^1_\rho(\R_{-}, \R_+) \to \R_+$. 

Clearly, the local asymptotical stability of the positive equilibrium does not exclude other and more complex (for example, periodic or unbounded) types of dynamical behaviour of trajectories for the semiflow $\frak F^t \phi$. This question was not considered in \cite{FBCD} and it is the main object of studies in the present work. Here, by invoking several ideas from \cite{HT} (where they were used to study the renewal equation associated  with a specific Gurtin-MacCamy's population model),  we describe the general structure of the semiflow  $\frak F^t \phi$ under relatively weak assumptions on functions $\beta$ and $g$. 

In our main result the monotonicity assumption {\bf (M)} is not required though similarly to \cite{FBCD} we impose the global Lipschitz condition on the function $\beta:[x_m,+\infty)\to [0, +\infty)$. 

\begin{thm}\label{mainTT} Assume that $0 \leq  \beta(x_m)< \mu$,  $x_m \geq 0$, $R(0) >1$,  that 
$\beta(x) > 0$ a.e. on $[x_0,+\infty)$  for some $x_0\geq x_m$,  that $\beta$ satisfies the global Lipschitz condition on $\R_+$ and  that positive continuous function $g$ vanishes at $+\infty$.   Then on the open set $L^1_\rho(\R_{-}, \R_+)\setminus \mathcal N$ the semiflow  $\frak F^t \phi$ has a compact global attractor $\mathcal A$ 
attracting each solution with initial datum in  $L^1_\rho(\R_{-}, \R_+)\setminus \mathcal N$.  Furthermore, restricted on any compact interval $[-n,0] \subset \R_-$, $n \in \N$, this convergence is uniform in the sup-norm.  Next, each element of $\mathcal A$ is a  differentiable function and there are universal positive numbers $\theta_1 \leq \theta_2, \theta_3$ such that   
$$
0 <\theta_1 \leq \psi(s) \leq \theta_2, \quad |\psi'(s)| \leq \theta_3, \ s \leq 0, \ \mbox{for each} \ \psi \in \mathcal A. 
$$
\end{thm}
In particular, by establishing good permanence properties of the semiflow $\frak F^t \phi$, 
this result additionally supports  the well-posedness of  model  (\ref{RE}). An explicit (constructive) upper bound $\theta_2$ for $\mathcal A$ in the sup-norm is given in Corollary \ref{cor11}.  According to the classification of global attractors in \cite[Chapter 2]{ST}, the compact set $\mathcal A$ in Theorem \ref{mainTT} is a global attractor of points. 
In view of the possibility of complete positive orbits connecting the zero equilibrium and $\mathcal A$, 
this attractor does not attract every bounded set. Thus $\frak F^t \phi$ restricted on the open subset  $L^1_\rho(\R_{-}, \R_+)\setminus \mathcal N$ is not asymptotically smooth, cf.  \cite[Theorem 2.33]{ST}.  Still, as the next theorem shows, 
$\mathcal A$ can be a compact attractor of neighbourhoods of compact sets, cf. \cite{MZh,ST}. 
In this theorem, similarly to \cite{FBCD} assuming  the hypothesis {\bf (M)}, we describe the structure of the 
global attractor $\mathcal A$. The key fact here is that {\bf (M)} implies the monotonicity of semiflow $\frak F^t \phi$  \cite{Smith} (see Section \ref{Ex1} for even more general result).  

\begin{thm}\label{mainT} Assume that $\mu >0, x_m \geq 0$, that $\beta$ is globally Lipschitzian function, and that  the assumption {\bf (M)} is satisfied.  Then 
\begin{itemize}
\item[(a)] the inequalities $R(0) >1$ and $\beta(x_m) < \mu$ guarantee that the unique positive equilibrium $b_*$ attracts all non-zero solutions and is locally stable in the uniform norm: i. e. for each $\epsilon >0$ there exists $\delta >0$ such that each solution $b(t,\phi)$ of the initial value problem $b(s,\phi)=\phi(s),\ s \leq 0,$ satisfies 
$|b(t,\phi)-b_*| < \epsilon$, $t \geq 0$, once $|\phi- b_*|_\rho < \delta$.  Moreover, if $\delta >0$ is sufficiently small then 
$\lim_{t \to +\infty}|b(t,\phi)-b_*|=0$ uniformly on the open ball $|\phi- b_*|_\rho < \delta$. In particular, $\mathcal A= \{b_*\}$ attracts neighbourhoods of compact sets in $L^1_\rho(\R_{-}, \R_+)\setminus \mathcal N$. 
\item[(b)] the inequalities $R(0) >1$ and $\beta(x_m) \geq \mu$ imply that each non-zero solution of equation  (\ref{RE}) is unbounded;
\item[(c)] $R(0) \leq 1$ guarantees that each solution of equation  (\ref{RE}) vanishes at $+\infty$.  
\end{itemize}
\end{thm}
Note that the type of local stability used in the statement (a) of the above theorem is stronger than that one used in   \cite{FBCD} where the closeness  $|b_t-\kappa|_\rho < \epsilon$, $t \geq0$,  in the weighted integral norm $|\cdot|_\rho$ is obtained instead of the closeness in the uniform norm. 

Finally, let us  add a few words about the organisation of the paper. In the next section, we analyse the monotonicity property 
of the operator ${\mathcal F}$.  Theorems \ref{mainTT} and \ref{mainT} are proved 
in the Sections \ref{UB}. The proofs of Propositions \ref{P1} and \ref{P2}  are outlined in Appendix.

\section{Monotonicity of the operator ${\mathcal F}$}  \label{Ex1} 

 \begin{thm}\label{Th1}
Suppose that $\beta$ is an increasing continuous function of polynomial growth and   $g:\R_+ \to \R_+$ is continuous  and bounded. 
Then the operator  ${\mathcal F}: L^1_\rho(\R_{-}, \R_+) \to \R_+$ defined by (\ref{OF}) is increasing with respect of the natural order in $L^1_\rho(\R_{-}, \R_+)$ for each $\rho < \mu$. 
\end{thm}
\begin{proof} First we assume that $\phi(s) >0$ a.e. on $\R_-$.  Set $h(s)=\phi(s)e^{\mu s}$ and $\theta = H(u)=\int_{-\infty}^uh(s)ds$,  a  straightforward computation shows  that 
$$
{\mathcal F}\phi=\int_0^\infty \beta\left( x_m + \int_0^a g\left( e^{-\mu(\tau -a)} \int_a^\infty e^{-\mu s} \phi(-s)\,ds \right) d\tau \right) e^{-\mu a} \phi(-a)da= 
$$
$$\int_{-\infty}^0 \beta\left( x_m + \int^0_u g( e^{-\mu s}H(u))ds\right)h(u)du =
\int_{0}^{H(0)}\beta\left( x_m + \int^0_{H^{-1}(\theta)} g( e^{-\mu s}\theta)ds\right)d\theta. 
$$
In general, the latter integral is an improper Riemann integral (can have a singularity at $\theta =0$) and $H(u)$ is strictly increasing and positive on $\R$,  with $H(-\infty)=0$.  

If $\phi_2(s) \geq \phi_1(s) >0$, almost everywhere on $\R_-$, then clearly 
$$
H_2(u)=\int_{-\infty}^uh_2(s)ds \geq H_1(u)=\int_{-\infty}^uh_1(s)ds, \quad u \leq 0, 
$$
so that 
$$
H_2^{-1}(\theta) \leq H_1^{-1}(\theta) \leq 0,    \quad 0 <\theta \leq H_1(0).  
$$
Consequently,  
$$ {\mathcal F}\phi_2 = \int_{0}^{H_2(0)}\beta\left( x_m + \int^0_{H_2^{-1}(\theta)} g( e^{-\mu s}\theta)ds\right)d\theta \geq 
$$
$$\int_{0}^{H_1(0)}\beta\left( x_m + \int^0_{H_1^{-1}(\theta)} g( e^{-\mu s}\theta)ds\right)d\theta = {\mathcal F}\phi_1.
$$
Now, if  $\phi_2(s) \geq \phi_1(s) \geq 0$, then the above argument shows that 
${\mathcal F}(\phi_2 +\delta) \geq {\mathcal F}(
\phi_1 +\delta)$ for each positive constant function $\delta$.  Since 
${\mathcal F}(\phi_j +\delta)$ depends continuously on $\delta$ we obtain that ${\mathcal F}\phi_2 \geq {\mathcal F}
\phi_1$ by considering $\delta \to 0^+$. 
This completes the proof of Theorem \ref{Th1}. \qed
\end{proof}
\begin{cor} For $b\geq 0$ set $F(b)= bR(b)$. If $\beta$ and $g$ satisfies all assumptions of Theorem \ref{Th1} and $\beta(x_m+x) >0$ for $x >0$ then $F(b)$ is a strictly increasing function. 
\end{cor}
\begin{proof} Indeed,  ${\mathcal F}b=bR(b)$ while for $b_2 > b_1$ we have $H_2(0) = b_2/\mu > H_1(0) = b_1/\mu$.  \qed 
\end{proof}
\begin{cor} \label{cor7} Assume that $b(t)$ is a bounded continuous solution of the integral equation (\ref{RE}) (i.e.  $b(t)={\mathcal F}b_t$, $ t \in \R$, where we use the standard notation $b_t(s)=b(t+s), s \leq 0$).  Set 
$m=\inf_{t \in \R} b(t) \leq M=\sup_{t \in \R} b(t)$.   If $\beta$ and $g$ satisfies all assumptions of Theorem \ref{Th1}, then $[m,M] \subseteq [F(m),F(M)]$: 
$$
F(m)\leq m\leq M \leq F(M). 
$$
\end{cor}

\section{Uniform ultimate boundedness and uniform persistence of solutions.  Proofs of Theorems \ref{mainTT} and  \ref{mainT}}\label{UB}
\begin{thm} \label{Th7}Assume that 
$\beta(x_m+x)\leq \beta(x_m)+cx$ for some $c >0$ and for all $x \geq 0$, $\beta(x_m)< \mu$,  that $g$ is a non-negative continuous function with  $g(+\infty)=0$.   Let $b(t)$ denote continuous solution of the initial value problem $b_0=\phi \in L^1_\rho(\R_{-}, \R_+)$ for equation  (\ref{RE}).  Then there are constant $\alpha_j>0$ depending only on $\beta, g, \mu$ such that
$$
\limsup_{t \to +\infty} b(t) \leq \alpha_1 \quad  \mbox{and}\quad  0 \leq  b(t) \leq \alpha_2|\phi|_\rho + \alpha_3, \quad t \geq 0. 
$$
\end{thm}
\begin{proof}
Consider  the Volterra form $b(t) =  (\mathcal{V}_\phi b)(t) + r_\phi(t)$ of  equation  (\ref{RE}). Note that
\begin{align*}
& (\mathcal{V}_\phi b)(t) = \int_0^t \beta\left( x_m + \int_0^{t-a} g\left( e^{-\mu(\tau -t+a)}\int_{t-a}^\infty e^{-\mu s}b(t-s)ds \right) d\tau \right) e^{-\mu(t-a)}b(a)da =\\
&e^{-\mu t}\int_0^t \beta\left( x_m+\int_0^{t-a} g\left( e^{-\mu(\tau -t+a)}\int_{-\infty}^a e^{-\mu(t-s)}b(s)ds \right) d\tau \right) e^{\mu a}b(a)da = \\
&e^{-\mu t}\int_0^t \beta\left( x_m + \int_a^t g\left( e^{-\mu\tau}\int_{-\infty}^a  e^{\mu s}b(s)ds \right) d\tau \right) e^{\mu a}b(a)da \leq \\
&e^{-\mu t} \int_0^t \left[ \beta(x_m) + c\int_a^t g\left( e^{-\mu \tau}\int_{-\infty}^a e^{\mu s}b(s)ds \right) d\tau \right] e^{\mu a}b(a)da =   \\
& \beta(x_m)\int_0^t e^{-\mu(t-a)} b(a)da  + ce^{-\mu t}\int_0^t\int_a^t g\left(e^{-\mu\tau}\int_{-\infty}^a  e^{\mu s}b(s)ds \right)e^{\mu a}b(a)d\tau da=: J_1+ J_2.
\end{align*}
Then setting  $G(x)=\int_0^x g(t)dt$, we find that 
\begin{align*}
J_2 &= ce^{-\mu t}\int_0^t \int_0^\tau g\left( e^{-\mu\tau}\int_{-\infty}^a e^{\mu s}b(s)ds \right) e^{\mu a}b(a)dad\tau= \\
&ce^{-\mu t}\int_0^t e^{\mu \tau}\int_0^\tau \frac{d}{da}\left[ G\left( e^{-\mu\tau}\int_{-\infty}^a e^{\mu s}b(s)ds \right) \right] da~d\tau=\\
&ce^{-\mu t}\int_0^t e^{\mu\tau} \left[ G\left( e^{-\mu\tau}\int_{-\infty}^\tau e^{\mu s}b(s)ds \right) - G\left( e^{-\mu\tau}\int_{-\infty}^0 e^{\mu s}b(s)ds \right) \right] d\tau =\\
&c\int_0^t e^{-\mu(t-\tau)}G\left( e^{-\mu\tau}|\phi|_{\mu} + \int_0^\tau e^{-\mu(\tau-s)}b(s)ds \right) d\tau - c\int_0^t e^{-\mu(t-\tau)}G(e^{-\mu\tau}|\phi|_{\mu})d\tau. 
\end{align*}
Next, 
\begin{align*}
&r_\phi(t)e^{\mu t} = \int_t^\infty \beta\left( x_m + \int_0^a g\left( e^{-\mu(\tau -a)} \int_a^\infty e^{-\mu s} \phi(t-s)ds \right) d\tau \right) e^{\mu(t- a)} \phi(t-a)da \\
	&=\int_{-\infty}^0 \beta\left( x_m + \int_a^t g\left( e^{-\mu\tau} \int_{-\infty}^a e^{\mu s}\phi(s)ds \right) d\tau \right) e^{\mu a}\phi(a)da  \leq \\
	&\beta(x_m)  \int_{-\infty}^0 e^{\mu a}\phi(a)da+ c\int_{-\infty}^0\int_a^t g\left( e^{-\mu \tau} \int_{-\infty}^a e^{\mu s}\phi(s)ds \right) e^{\mu a}\phi(a)d\tau da =
	\end{align*}
	\begin{align*}
	& \beta(x_m)|\phi|_\mu  + c\int_0^t \int_{-\infty}^0 g\left( e^{-\mu \tau} \int_{-\infty}^a e^{\mu s}\phi(s)ds \right) e^{\mu a}\phi(a)dad\tau \\
	& \quad + c \int_{-\infty}^0 \int_{-\infty}^\tau g\left( e^{-\mu \tau} \int_{-\infty}^a e^{\mu s}\phi(s)ds \right) e^{\mu a}\phi(a)~dad\tau=\\
	& \beta(x_m) |\phi|_\mu  + c\int_0^t e^{\mu \tau} G(e^{-\mu\tau}|\phi|_\mu)d\tau
	+ c\int_{-\infty}^0 e^{\mu\tau} G\left( e^{-\mu\tau} \int_{-\infty}^\tau e^{\mu s}\phi(s)ds \right)d\tau,
\end{align*}
because of  $\lim_{a\to -\infty} G\left( e^{-\mu\tau} \int_{-\infty}^a e^{\mu s}\phi(s)~ds \right) = G(0)=0$. 

Consequently, 
$$
	e^{\mu t}b(t) \leq A+ \beta(x_m)\int_0^t e^{\mu \tau} b(\tau)d\tau + c\int_0^t e^{\mu\tau}G\left(e^{-\mu\tau}|\phi|_\mu + e^{-\mu\tau}\int_0^\tau e^{\mu s}b(s)ds  \right)d\tau,
$$
where 
$$A= \beta(x_m)|\phi|_\mu + c\int_{-\infty}^0 e^{\mu \tau} G\left( e^{-\mu\tau} \int_{-\infty}^\tau e^{\mu s}\phi(s)ds \right)d\tau.
$$
Setting $B(t)= e^{\mu t}b(t)$, we can rewrite the latter inequality  as 
\[
	B(t) \leq A + \beta(x_m)\int_0^t B(\tau)d\tau + c\int_0^t e^{\mu \tau} G\left( e^{-\mu \tau}|\phi|_\mu + e^{-\mu \tau}\int_0^\tau B(s)ds \right)d\tau.
\]
Now, since $g(+\infty) =0$, we find that $G(x)/x \to 0$ as $x\to \infty$. This implies that for each $\varepsilon\in (0,1)$ there exists a constant $C_\varepsilon$ such that $G(x) \leq C_{\varepsilon} + \varepsilon x$ for all $x\geq 0$. 
Thus 
\begin{align*}
	&\int_0^t e^{\mu \tau} G\left( e^{-\mu \tau}|\phi|_\mu + e^{-\mu \tau}\int_0^\tau B(s)ds \right)d\tau \leq \int_0^t e^{\mu \tau} \left[ C_\varepsilon + \varepsilon e^{-\mu \tau}|\phi|_\mu + \varepsilon e^{-\mu \tau}\int_0^\tau B(s)ds \right] d\tau\\
	& = {C_\varepsilon}\mu^{-1}(e^{\mu t}-1) + \varepsilon |\phi|_\mu t + \varepsilon \int_0^t \int_0^\tau B(s)dsd\tau,
\end{align*}
$$
A   \leq 
	\left(\beta(x_m)+
\frac{c}{\mu -\rho}\right)|\phi|_\rho + \frac{cC_\varepsilon}{\mu}, 
$$
so that 
\[
	B(t) \leq A + \varepsilon c|\phi|_\mu  t + cC_\varepsilon \mu^{-1} (e^{\mu t}-1) + \beta(x_m)\int_0^t B(\tau)d\tau + \varepsilon c \int_0^t \int_0^\tau B(s)dsd\tau =
\]
\[
	A + \varepsilon c|\phi|_\mu  t + cC_\varepsilon \mu^{-1} (e^{\mu t}-1)  + \int_0^t [\beta(x_m) + \varepsilon c(t-s)] B(s)ds, \quad t \geq 0.
\]
Iterating this inequality where the right-hand side define a Volterra type operator, we find that 
 $B(t) \leq \xi(t)$, where $\xi(t)$ is a positive function  satisfying the integral equation
\[
	\xi(t) = A + \varepsilon c|\phi|_\mu  t + cC_\varepsilon \mu^{-1} (e^{\mu t}-1)  + \int_0^t [\beta(x_m) + \varepsilon c(t-s)] \xi(s)ds, \quad t \geq 0. 
\]
Clearly, $\xi(t)$ also solves the initial value problem
\begin{equation*}
	\begin{cases} 	\xi''(t) = \mu c C_\varepsilon e^{\mu t} + \beta(x_m)\xi'(t) + \varepsilon c\xi(t) \\
	 \xi(0)=A,~ \xi'(0)=\varepsilon c|\phi|_\mu  + c C_\varepsilon + \beta(x_m)A.
	 \end{cases}
\end{equation*}
Solving this problem we obtain that
\[
	\xi(t) = \alpha_1 e^{\mu t} + a_1 e^{\lambda_+ t} + a_2 e^{\lambda_{-}t},
\]
where 
  $$\alpha_1 = \frac{\mu c C_\varepsilon}{\mu^2 - \mu \beta(x_m) - \varepsilon c}, \quad 
	\lambda_\pm = \frac{\beta(x_m) \pm \sqrt{\beta^2(x_m) + 4\varepsilon c}}{2},
$$
and $a_1,~a_2$ are constants determined from the linear system
$$ a_1 + a_2 =A- \alpha_1,~  \lambda_+ a_1 + \lambda_-a_2 =-\alpha_1\mu +\varepsilon c|\phi|_\mu  + c C_\varepsilon + \beta(x_m)A.
 $$
Since $\beta(x_m)< \mu$ we can take $\varepsilon>0$ small enough to have $\lambda_{-}<\lambda_+ < \mu$. 
Thus we conclude that
$$
	b(t) \leq \alpha_1 + a_1 e^{(\lambda_+ -\mu)t} + a_2 e^{(\lambda_{-} - \mu)t} \leq \alpha_1 + |a_1|+|a_2| \leq \alpha_2|\phi|_\rho + \alpha_3, \quad t \geq 0, 
$$
$$\limsup_{t \to +\infty} b(t) \leq \alpha_1= \frac{\mu c C_\varepsilon}{\mu^2 - \mu \beta(x_m) - \varepsilon c}, $$
with some $\alpha_j$ independent of the initial data $\phi$.  \qed
\end{proof}
\begin{cor} \label{cor8}Suppose that $\beta(x)$ is a globally Lipschitzian function with the Lipschitz constant $L$ (so that $\beta(x_m+x)\leq \beta_1(x):=\beta(x_m)+Lx$ for all $x \geq 0$), that $\beta(x_m) < \mu$ and that  $g$ is a non-negative continuous function with  $g(+\infty)=0$.    Then there are constant $\alpha_4, \alpha_5>0$ depending only on $\beta, g, \mu$ and such that continuous solution $b(t)$ of the initial value problem $b_0=\phi \in L^1_\rho(\R_{-}, \R_+)$ for equation  (\ref{RE}) is differentiable on $(0,+\infty)$ and has  uniformly bounded derivative: 
\begin{equation}\label{esty}
|b'(t)| \leq \alpha_4|\phi|_\rho + \alpha_5, \quad t > 0. 
\end{equation}
\end{cor} 
\begin{proof} Note that
$$
b(t) = \int_{-\infty}^t \beta\left(x_m + \int_a^t g\left( e^{-\mu \tau} \int_{-\infty}^a e^{\mu s} b(s)ds \right) d\tau \right) e^{-\mu(t-a)}b(a)da.
$$
Since $|\beta'(x_m+x)|\leq L$ a.e. on $\R_+$ and $\phi \in L^1_{\rho}(\R_{-})$,  by differentiating the last equality, we obtain
\begin{align*}
&b'(t) = \beta(x_m)b(t) + \int_{-\infty}^t \frac{d}{dt}\left[ \beta\left(x_m + \int_a^t g\left( e^{-\mu \tau} \int_{-\infty}^a e^{\mu s} b(s)ds\right) d\tau \right) e^{-\mu(t-a)}b(a) \right] da \\
&= \beta(x_m)b(t) -\mu b(t) +\\
&  \int_{-\infty}^t \beta'\left( x_m + \int_a^t g\left( e^{-\mu \tau} \int_{-\infty}^a e^{\mu s}b(s)ds\right)d\tau\right)g\left( e^{-\mu t}\int_{-\infty}^a e^{\mu s}b(s)ds\right) e^{-\mu (t-a)}b(a)da.
\end{align*}
Thus, using the estimate $ 0 \leq  b(t) \leq \alpha_2|\phi|_\rho + \alpha_3, \quad t \geq 0,$ we find that 
\begin{align*}
	|b'(t)| &\leq (\mu-\beta(x_m))(\alpha_2|\phi|_\rho + \alpha_3) + L\max_{x \geq 0}g(x)e^{-\mu t} \left[ |\phi|_\rho + \int_0^t e^{\mu a}b(a)da \right]\\
	&\leq (\mu-\beta(x_m))(\alpha_2|\phi|_\rho + \alpha_3) +  L\max_{x \geq 0}g(x)\left[ |\phi|_\rho + \frac{(\alpha_2|\phi|_\rho + \alpha_3)}{\mu} \right], 
\end{align*}
and the estimate (\ref{esty}) follows. \qed
\end{proof}
Now, consider the operator ${\mathcal F_1}$ defined by (\ref{OF}) where $\beta(x_m+x)$ is replaced with $\beta_1(x)= \beta(x_m)+Lx \geq \beta(x+x_m), x \geq 0$ (cf. Corollary \ref{cor8}) and $g(+\infty)=0$. Clearly, 
${\mathcal F}\phi \leq {\mathcal F_1}\phi$ for each $\phi \in L^1_\rho(\R_{-}, \R_+)$. By Theorem \ref{Th1}, the functional  ${\mathcal F_1}$ is monotone on  $L^1_\rho(\R_{-}, \R_+)$ and the monotone scalar function 
$F_1(c)= {\mathcal F_1}c$ (here $c$ is considered as a constant element of $L^1_\rho(\R_{-}, \R_+)$) has the representation $F_1(c)= cR_1(c)$ where 
$$
R_1(c):=\int_0^\infty \left(\beta(x_m) + L\int_0^a g\left( \frac{e^{-\mu\tau}}{\mu}c\right) d\tau \right) e^{-\mu a}da. $$
We have that  $
R_1(0) \geq R(0) >1, \ R_1(+\infty) = R(+\infty) =\beta(x_m)/\mu < 1, 
$
so that $F_1(c)$ has the biggest positive equilibrium $\theta_2$.  

\begin{cor} \label{cor11} Assume all conditions of Corollary \ref{cor8}. 
If $R(0) >1$ then  each continuous non-zero bounded function $b(t)$ satisfying equation (\ref{RE}) for all $t \in \R$ 
has $\theta_2$ as an upper bound: $
0\leq  b(t) \leq  \theta_2, \quad t \in \R. 
$
\end{cor}
\begin{proof} Without loss of generality, after using an appropriate translation or a limiting argument, we can assume  that $b(0)= \max_{s \in \R}b(s)$. But then $b(0) \leq \theta_2$ as a consequence of the chain of inequalities  
$b(0)= {\mathcal F}b_0 \leq {\mathcal F_1}b_0 \leq {\mathcal F_1}b(0)= F_1(b(0)). $ \qed
\end{proof}

\begin{cor} \label{cor9} Assume all conditions of Corollary \ref{cor8} and that $g(x) >0$ for all $x\geq 0$, 
$\beta(x) > 0$ a.e. on $[x_0,+\infty)$  for some $x_0\geq x_m$. 
If $R(0) >1$ then  the semiflow $\frak F^t \phi$  is uniformly persistent: i.e. there exists $\theta_1 >0$ such that  $\liminf_{t\to \infty} b(t)\geq \theta_1$ for each non-zero solution $b(t)$. 
\end{cor}

\begin{proof} Let $C_b$ denote the space of all continuous bounded functions $c:\R_+\to \R_+$ endowed with the sup-norm. 
The functional ${\mathcal R}: C_b \to \R_+$ defined by 
$$
{\mathcal R}(c)=\int_0^\infty \beta\left( x_m + \int_0^a g\left( \frac{e^{-\mu\tau}}{\mu}c(a)\right) d\tau \right) e^{-\mu a}da$$
is clearly continuous at $c=0$ and therefore there exists some $\delta >0$ such that 
$$\rho:=  \inf\left\{{\mathcal R}(c): c\in C_b, \ c(a)\in [0,\delta], a \geq 0 \right\}>1. $$  


We start by proving the impossibility of the convergence $\lim_{t\to \infty} b(t)=0$ for an eventually positive solution $b(t)$. 
Due to Proposition \ref{P2}, without loss of generality we can assume that
$b(t) >0$ for all $t \geq 0$. 
If $b(+\infty)=0$, there exists  an  increasing sequence of numbers $t'_j\to \infty$,  such that $b(t'_j)\to 0$ and $0<b(t'_j)\leq b(t)$ for all $t \in [0,t'_j]$.  Then we get from (\ref{RE}) that
$$
b(t'_j) \geq \int_0^{t'_j} \beta\left( x_m + \int_0^a g\left( e^{-\mu(\tau -a)} \int_a^\infty e^{-\mu s} b(t'_j-s)\,ds \right) d\tau \right) e^{-\mu a} b(t'_j-a)da \geq 
$$
$$
 \int_0^{t'_j} \beta\left( x_m + \int_0^a g\left( e^{-\mu(\tau -a)} \int_a^\infty e^{-\mu s} b(t'_j-s)\,ds \right) d\tau \right) e^{-\mu a} b(t'_j)da. 
$$
Consequently, 
\begin{equation}\label{ab}
1 \geq  \int_0^{t'_j} \beta\left( x_m + \int_0^a g\left( e^{-\mu(\tau -a)} \int_a^\infty e^{-\mu s} b(t'_j-s)\,ds \right) d\tau \right) e^{-\mu a}da. 
\end{equation}

Taking the limit in the latter  inequality as $j \to \infty$, we obtain that
\[
	1 \geq \int_0^\infty \beta(x_m + ag(0))e^{-\mu a}da = R(0) >1,
\]
a contradiction.  Consequently, $\limsup_{t\to \infty} b(t)>0$ for each non-zero solution $b(t)$.

Let assume now that   the semiflow $\frak F^t \phi$  is not uniformly persistent, i.e. 
for each $n$ there exists a non-zero solution $b_n(t)$ such that  $\liminf_{t\to \infty} b_n(t)< 1/n$. Using, if necessary, appropriately shifted solutions $b_n(t+s_n)$ with $s_n \geq 0$  and invoking the ultimate uniform boundedness 
result of Theorem \ref{Th7}, without loss of generality we can assume that $b_n(t) >0$ for all $t  \geq 0$ and that the sequence of the initial segments $\phi_n(s)=b_n(s), \ s \leq 0$, is bounded in  $L^1_\rho(\R_{-}, \R_+)$. In view of (\ref{esty}), this implies the uniform boundedness  of the sequence $\{b'_n(t)\}$ on $\R_+$.

We will consider the following two situations: 

\underline{Case 1}. For some $m \in \N$ it holds that 
\begin{equation}\label{ca1}
0\leq  \zeta_1 =:\liminf_{t\to \infty} b_m(t) \leq \limsup_{t\to \infty} b_m(t):=\zeta_2\leq \delta, \quad \zeta_2 >0. 
\end{equation}
If $ \zeta_1=0$ then there exists a monotone sequence $s_k \to +\infty$ such that $\lim_{k\to \infty} b_m(s_k)=0$
and  $0<b_m(s_k)\leq b_m(t)$ for all $t \in [0,s_k]$.  Since equation (\ref{RE}) is translation invariant, the functions 
$b_{m,j}(t):=b_m(t+s_j)$ are also solutions of it. By invoking the Arzel\'a-Ascoli theorem and using Theorem \ref{Th7} and Corollary \ref{cor8}, we may conclude that some subsequence $b_{m,j_k}(t)$ is converging, uniformly on compact subsets of $\R$,  to a continuous function $b_*(t)$ such that $b_*(t)\leq  \zeta_2\leq \delta$, $t \in \R$.  Arguing as above of the formula (\ref{ab}), we also obtain 
$$
1 \geq  \int_0^{s_{j_k}} \beta\left( x_m + \int_0^a g\left( e^{-\mu(\tau -a)} \int_a^\infty e^{-\mu s} b_m(s_{j_k}-s)\,ds \right) d\tau \right) e^{-\mu a}da. 
$$
After taking the limit in this  inequality as $k \to \infty$, we get the following contradiction 
$$ 
1 \geq \int_0^\infty \beta\left( x_m + \int_0^a g\left( e^{-\mu(\tau -a)} \int_a^\infty e^{-\mu s} b_*(-s)\,ds \right) d\tau \right) e^{-\mu a}da >1,
$$
since
\begin{equation}\label{bo}
e^{-\mu(\tau -a)} \int_a^\infty e^{-\mu s} b_*(-s)\,ds \leq \frac{\delta}{\mu}e^{-\mu\tau}, \quad a \geq 0. 
\end{equation}
If $ \zeta_1>0$, we can modify the previous argumentation as follows. Take $\epsilon \in (0, 1)$ small enough 
to satisfy $(1-\epsilon)\rho >1$. We also can assume that $b_m(t) \geq \zeta_1(1-\epsilon)$ for all $t \geq 0$. 
Then there exists a monotone sequence $s_k \to +\infty$ such that $\lim_{k\to \infty} b_m(s_k)=\zeta_1$. Choose a continuous function $b_*(t)$ as above. Then 
$$
 b_m(s_k)/(\zeta_1(1-\epsilon)) \geq  \int_0^{s_k} \beta\left( x_m + \int_0^a g\left( e^{-\mu(\tau -a)} \int_a^\infty e^{-\mu s} b_m(s_k-s)\,ds \right) d\tau \right) e^{-\mu a}da. 
$$
After taking the limit in this  inequality as $k \to \infty$, we get the following contradiction 
$$ 
1/(1-\epsilon) \geq \int_0^\infty \beta\left( x_m + \int_0^a g\left( e^{-\mu(\tau -a)} \int_a^\infty e^{-\mu s} b_*(-s)\,ds \right) d\tau \right) e^{-\mu a}da \geq \rho,
$$
in view of (\ref{bo}).
Thus the situation presented in formula  (\ref{ca1}) cannot happen.

\underline{Case 2}. For each $n \in \N, \ n > 1/\delta$,  it holds that 
$$0\leq  \zeta_n =:\liminf_{t\to \infty} b_n(t) < 1/n < \delta <\limsup_{t\to \infty} b_n(t).$$ 

Then for each fixed $n> 1/\delta$ there are increasing sequences of positive numbers $\{s_j\}, \{t_j\},$ $s_j < t_j$ such that $s_j, t_j\to \infty$,  $b_n(s_j)=\delta$, $b_n(t_j) = 1/n$ and $1/n\leq b_n(t) \leq \delta$ for all $t \in [s_j,t_j]$. Choose some converging subsequence  $t_{j_k}-s_{j_k} \to \tau_n \in [0,+\infty]$. Since equation (\ref{RE}) is translation invariant, the functions 
$b_{n,j}(t)=b_n(t+s_j)$ are also solutions of it. By invoking the Arzel\'a-Ascoli theorem and using Theorem \ref{Th7} and Corollary \ref{cor8}, we may conclude that some subsequence $b_{n, j_k}(t)$ is converging, uniformly on compact subsets of $\R$,  to  a continuous bounded function $d_n(t)$ such that $d_n(0)= \delta= \max_{s \in [0, \tau_n]}d_n(s)$ and $\limsup_{t\to \infty} d_n(t) >0$. But then the option $\tau_n= +\infty$ is not possible in view of the analysis done in the Case 1.  Therefore $\tau_n$ is a finite positive number and $d_n(\tau_n)=1/n$. 
Clearly, the sequences  $d_n(t)$ and  $d'_n(t)$ are uniformly bounded on $\R$.  Therefore, once more applying  the Arzel\'a-Ascoli theorem, we find some subsequence $i_k$ and an extended positive number $\tau_\star \in [0, +\infty]$ such that 
$\tau_{i_k} \to \tau_\star$ and  $d_{i_k}(t)$ is converging, uniformly on compact subsets of $\R$,  to  a continuous bounded function $d_\star(t)$ such that $d_\star(0)= \delta= \max_{s \in [0,\tau_\star]}b(s)$ and $\limsup_{t\to \infty} d_\star(t) >0$. Again, the option $\tau_\star= +\infty$ is not possible in view of the analysis done in the Case 1.  Therefore $\tau_\star$ is a finite positive number and $d_\star(\tau_\star)=0$. However, since $d_\star(t)$ satisfies (\ref{RE}) for all $t \in \R$ and $d_\star(0)= \delta >0$ we know that $d_\star(t) >0$ for all $t \in \R$. The obtained contradiction completes the proof of  Corollary \ref{cor9}.  \qed 
\end{proof}
{\sc \bf Proof of Theorem \ref{mainTT}.} Let $\mathcal B$ comprises all continuous functions $\phi: (-\infty,0] \to [0,+\infty)$ such that $0.5\theta_1 \leq \phi(t) \leq 2\theta_2,\ t \leq 0,$ with $\theta_j$ defined in Corollaries \ref{cor11} and \ref{cor9}. Then  $\mathcal B$ 
is bounded  in $L^1_\rho(\R_{-}, \R_+)$ by  $2\theta_2/\mu$. Next, consider the subset $\mathcal D$ of $\mathcal B$  consisting from all lipschitzian  functions with the Lipschitz constant which is less or equal to  $\theta_3:= 2\alpha_4\theta_2/\mu + \alpha_5$ with  $\alpha_4, \alpha_5$ being defined in Corollary \ref{cor8}. Clearly,  $\mathcal D$ 
is a compact subset of $L^1_\rho(\R_{-}, \R_+)$ attracting every positive solution due to Corollaries \ref{cor8}, \ref{cor11} and \ref{cor9}. This fact  implies the existence of the compact attractor $\mathcal A$ of points (e.g. see \cite[Theorem 2.17(c)]{ST}) possessing properties mentioned in the statement of Theorem \ref{mainTT}.   \qed 

\vspace{2mm}

\noindent{\sc \bf Proof of Theorem \ref{mainT}.}  (a) Set $\theta_1\leq m= \inf_{\phi \in \mathcal A}\phi(0); \ M= \sup_{\phi \in \mathcal A}\phi(0) $. Clearly, $\theta_1\leq m \leq M \leq \theta_2$ and,  
by Corollary \ref{cor7}, we find that $[m,M] \subseteq [F(m),F(M)]$. In view of the monotonicity properties of both $F(x)$ and $R(x)$, this can happens if and only if $m=M=b_*$. Thus $\mathcal A= \{b_*\}$. Let now suppose  that the equilibrium $b_*$ is not stable as indicated in the statement of the theorem. Then there exists $\epsilon >0$
a sequence $\phi_n \to b_*$ such that corresponding solutions $b_n(t)$ satisfy the estimate 
$
\sup_{t \geq 0}|b_n(t)-b_*| > \epsilon, $ $n \in \N. 
$
Since $\phi_n \to b_*$, the sequences $b_n(t), b'_n(t)$ are uniformly bounded on $\R_+$ and  $b_n(t)$ converges uniformly to $b_*$ on the intervals $[0,m]$, $m \in \N$.  Since $b_n(+\infty) = b_*$, there exists a sequence $s_n$ of positive numbers such that 
\begin{equation}\label{aa}
\max_{t \in [0,s_n]}|b_n(t)-b_*| = |b_n(s_n)-b_*|= \epsilon.
\end{equation}
 It is clear that $s_n \to \infty$.  But then, considering a sequence of shifted solutions $b_n(t+s_n)$ and using limiting argument, we establish the existence of a bounded continuous function $b_*(t)$ satisfying 
equation (\ref{RE}) for all $t \in \R$ and such that $$\sup_{t \in \R}|b_*(t)-b_*| \geq |b_*(0)-b_*| = \epsilon=\max_{t \leq 0}|b_*(t)-b_*| $$ This means that either $M_*=\sup_{t \in \R}b_*(t)$ or $m_*=\inf_{t \in \R}b_*(t)$ is different from $b_*$. Moreover, since $b_*(+\infty)=b_*$, we find that $m_*>0$. But the inclusion $[m_*,M_*] \subseteq [F(m_*),F(M_*)]$ then implies $m_*=M_*=b_*$, a contradiction proving the stability of $b_*$. 

Next, choose $\delta >0$ small enough to have $\theta_1 \leq b(t, \phi) \leq \theta_2, \ t \geq 0$, for each initial function $\phi$ satisfying $|\phi-b_*|_\rho \leq \delta$.  We claim that $b(t, \phi)$ converges to $b_*$ uniformly on  the  ball 
 $B_\delta= \{\phi:  |\phi-b_*|_\rho \leq \delta\}$. Indeed, otherwise there is $\gamma >0$ and a sequences $\phi_n \in B_\delta$ and $s_n \to + \infty$ such that  $|b(s_n,\phi_n)-b_*|=\gamma$. Then arguing as in the previous paragraph below (\ref{aa}), we obtain a contradiction.  Hence, $\mathcal A= \{b_*\}$ attracts its neighbourhood $B_\delta$ and therefore it also attracts neighbourhoods of compact sets in $L^1_\rho(\R_{-}, \R_+)\setminus \mathcal N$, e.g. see 
 \cite[Corollary 2.32]{ST}.

(b) Assume that $R(0) >1$ and $\beta(x_m) \geq \mu$, then $b'(t) = (\beta(x_m)-\mu) b(t) +$
$$ \int_{-\infty}^t \beta'\left( x_m + \int_a^t g\left( e^{-\mu \tau} \int_{-\infty}^a e^{\mu s}b(s)ds\right)d\tau\right)g\left( e^{-\mu t}\int_{-\infty}^a e^{\mu s}b(s)ds\right) e^{-\mu (t-a)}b(a)da, 
$$
so that $b(t)$ is an increasing function. If $b(+\infty)$ were finite, the constant function $b(t)=b(+\infty)$ should be an equilibrium for the equation (\ref{RE}) and therefore it should satisfy the equation $R(b(+\infty)) =1$ which is not possible. 

(c) Again, if $R(0)\leq 1$ then each solution is eventually bounded (cf. Corollary \ref{cor11}) while $[m,M] \subseteq [F(m),F(M)]$ implies that $m=M=0$. Observe  that $b_*=0$ when  $R(0)=1$.  \qed

\section{Conclusion}
In recent study \cite{FBCD}, C. Barril et al proposed a phenomenological size-structured model describing  the dynamics of trees growth in the forest. This model is given in the form of relatively complex integral renewal equation with infinite delay. Importantly, it nicely agrees  with the classical PDE formulation derived by invoking a conservation law \cite{FBCD}. 

In this manuscript, which can be considered as a companion paper for \cite{FBCD}, we provide additional arguments justifying the well-posedness of the mentioned renewal equation. In particular, under more realistic conditions imposed on the growth and reproduction rates $g$ and $\beta$ we establish the ultimate uniform boundedness and uniform persistence of  tree's population birth rate. In the case when $g$ and $\beta$ are monotone and the associated basic reproduction number is bigger than 1, we also prove the global asymptotic stability of the model. We believe that this result can be further extended for certain unimodal functions $\beta$, cf. \cite{HT}. 

It should be noted that our approach is rather different from the principle of linearised stability for delay equations used in \cite{FBCD}.  Here we follow the techniques developed in \cite{HT}  to study some renewal equation associated  with a specific Gurtin-MacCamy's population model.

\section{Appendix}
{\sc \bf Proof of Proposition \ref{P1}.} 
Note that the condition of polynomial growth of $\beta$ is sufficient to ensure that $r_\phi(t)$ is well defined for all $t\geq 0$. Indeed, we have that
\begin{align*}
	r_\phi(t) &= \int_{-\infty}^0 \beta\left( x_m+ \int_a^t g\left( e^{-\mu \tau} \int_{-\infty}^a e^{\mu s}\phi(s)ds \right) d\tau \right) e^{-\mu(t-a)}\phi(a)da\\
	&= e^{-\mu t}\int_{-\infty}^0 \beta\left( x_m+ \int_a^t g\left( e^{-\mu \tau} \int_{-\infty}^a e^{\mu s}\phi(s)ds \right) d\tau \right) e^{(\mu-\rho)a}e^{\rho a}\phi(a)da,
\end{align*}
so that the latter integral is finite. Moreover, for some $m \in \N$, $r_\phi(t)= O(t^me^{-\mu t})$ at $t\to\infty$.

{Part (a)} Let $C([0,T])_+$ denote the space of all non-negative continuous functions on the interval $[0,T]$ provided with the topology of uniform convergence.  Fix $T>0$, $\phi\in L^1_\rho(\R_{-},\R_+)$ and define $\mathcal{A} : C([0,T])_+ \to C([0,T])_+$ by $\mathcal{A}b = \mathcal{V}_\phi b + r_\phi$. It is straightforward  to check that  $\mathcal{A}$ satisfies conditions of the Schauder fixed point theorem. Moreover,  there exists norm $\|\cdot\|$ in $C([0,T])$ generating the topology of  uniform convergence and $R_0=R_0(T)>0$ for which $\mathcal{A} B_R \subseteq B_R$ for every $R\geq R_0$, where $B_R\subseteq C([0,T])_+$ denotes the subset of functions with norm less or equal to $R$. Thus equation 
$b = \mathcal{V}_\phi b + r_\phi$ has at least one solution $b= b(\phi)$ in $B_R$. 

Next, for each $n\in\N$ consider  continuous solution $b_n$  to \eqref{RE} existing on the finite interval $[0,n]$.  For every fixed $N\in \N$ the sequence $\{ b_n \}_{n\geq N}$ is uniformly  bounded and equicontinuous in $C([0,N])$. In fact, by taking $\theta>0$ large enough and considering the  norm $\| b \|_\theta = \max_{t\in [0,N]} |e^{-\theta t}  b(t)|$ we can prove that there is positive number $\alpha$ such that $\|b_n\|_\theta \leq \alpha \max_{t\geq 0} r_\phi(t)$ for every $n\geq N$. This property of the family $\{ b_n \}_{n\geq N}$ assures its equicontinuity. Then the Arzel\'a-Ascoli theorem allows to construct a continuous solution of \eqref{RE} on $[0,\infty)$ by means of a diagonal argument.

{Part (b).} Fix again $T>0$ and $\phi \in  L^1_\rho(\R_{-},\R_+)$. Since the argument of $\beta$ in the expression for $(\mathcal{V}_\phi b)(t)$ is changing between $x_m$ and $x_m+T\sup_{s \geq 0}g(s)$, without loss of generality we may suppose that $\beta$ is a globally Lipshitzian function. If $b_1(t), b_2(t)$ solve equation (\ref{RE})
on $[0,T]$ with the same initial datum $b_2(s)=b_1(s)=\phi(s), s \leq 0$, then it follows from (a) and the proof of Corollary \ref{cor8} that $|b_j(t)|+|b'_j(t)| \leq M,$ $t \in [0,T]$, with some universal constant $M$.  It can be also proved that 
$$|b_2(t)-b_1(t)| \leq \int_0^t K(t,s)|b_2(s)-b_1(s)|ds, \quad t \in [0,T],$$ 
with some non-negative function $K(t,s)$ (depending on $b_j(t)$)   which is continuous on the closed triangle $\{(s,t): 0\leq s \leq t \leq T\}$. 
Then an application of the Gronwall-Bellman inequality shows that $b_1(t)\equiv b_2(t)$ on $[0,T]$. 
The cases $\phi=0$ and $\phi \not=0$ were considered separately by us. The first case is trivial and in the latter case  we avoid imposing the Lipschitz condition on $g$ by using the following estimates (where we set 
$B_j(a)=  \int_{-\infty}^a e^{\mu s}b_j(s)ds$) 
$$
\left|\int_a^t g\left( e^{-\mu \tau} \int_{-\infty}^a e^{\mu s}b_2(s)ds \right) d\tau - \int_a^t g\left( e^{-\mu \tau} \int_{-\infty}^a e^{\mu s}b_1(s)ds \right) d\tau\right| = 
$$
$$
\frac 1 \mu\left| \int_{e^{-\mu a}B_2(a)}^{e^{-\mu t}B_2(a)}(g(s)/s)ds- \int_{e^{-\mu a}B_1(a)}^{e^{-\mu t}B_1(a)}(g(s)/s)ds\right| =
$$ 
$$
\frac 1 \mu\left|- \int_{e^{-\mu a}B_1(a)}^{e^{-\mu a}B_2(a)}(g(s)/s)ds+ \int_{e^{-\mu t}B_1(a)}^{e^{-\mu t}B_2(a)}(g(s)/s)ds\right| \leq 
$$ 
$$
\frac{2\sup_{s \geq 0}g(s)}{ \mu}\left|\ln\frac{B_2(a)}{B_1(a)}\right| \leq \frac{2\sup_{s \geq 0}g(s)}{ \mu}\left|\ln\left(1+ \frac{\int_{0}^a e^{\mu s}|b_2(s)-b_1(s)|ds}{B_1(0)}\right)\right|\leq 
$$ 
$$
\frac{2\sup_{s \geq 0}g(s)}{ \mu |\phi|_\mu}\int_{0}^t e^{\mu s}|b_2(s)-b_1(s)|ds, \quad 0 \leq a \leq t. 
$$ 

{Part (c).}  To prove that the solution $b_\phi$ depends continuously 
 on $\phi \in L^1_\rho(\R_{-},\R_+)$, take an arbitrary initial function $\phi_0$ in this space and consider sequence  $\phi_n \to \phi_0$ together with the sequence of respective solutions $b_n(t)$ and $b_0(t)$ for equation (\ref{RE}) on $[0,T]$. If the sequence $b_n(t)$ does not converge uniformly to $b_0(t)$ then it contains a subsequence $b_{n_j}(t)$
converging uniformly to some $b_*(t) \not= b_0(t)$, see proofs in (a), (b). Then  taking limit as $j \to +\infty$ in  the renewal equations (\ref{RE}) for $b_{n_j}(t)$, we obtain that the limiting function $b_*(t)$ also solves (\ref{RE}) with the initial data $\phi_0$. This, however, contradicts to the uniqueness result established in (b). 
 \qed

\vspace{2mm}

{\sc \bf Proof of Proposition \ref{P2}.} 
 Suppose now that $b(s_0) >0$ at some point $s_0>0$. Using the shifted solution $b_1(t)=b(t+s_0)$ if necessary, without loss of generality we can assume that $s_0=0$. 
Considering equation (\ref{RE}) in the Volterra's form $b(t) =  (\mathcal{V}_\phi b)(t) + r_\phi(t)$, 
we obtain that 
$b(t) \geq \int_0^t K(a)b(t-a)ds+ r_\phi(t)$, where 
$$
K(a)=  \inf_{t \geq a\geq 0}\beta\left( x_m + \int_0^a g\left( e^{-\mu(\tau -a)}(\int_a^t e^{-\mu s} b(t-s)ds+ e^{-\mu t}|\phi|_\mu)\right)d\tau\right) e^{-\mu a}$$
$$
r_\phi(t)= \int_t^\infty \beta\left( x_m + \int_0^a g\left( e^{-\mu(\tau -a)} \int_a^\infty e^{-\mu s} \phi(t-s)ds\right) d\tau \right) e^{-\mu a} \phi(t-a)da.
$$
Since for all $t \geq a \geq \tau \geq 0$, it holds that 
$$0 \leq e^{-\mu(\tau -a)}\int_a^t e^{-\mu s} b(t-s)ds \leq \Gamma:=\sup_{s \geq 0}b(s)/\mu +|\phi|_\mu, $$
if $b(t)$ is bounded on $\R_+$
we find that in every case 
$$\int_0^a g\left( e^{-\mu(\tau -a)}(\int_a^t e^{-\mu s} b(t-s)ds+ e^{-\mu t}|\phi|_\mu)\right)d\tau \geq ag_*,$$
where 
$
g_*:= \min_{x \in [0,\Gamma]}g(x) >0
$
if $b$ is bounded and $
g_*:= \inf_{x \geq 0}g(x)>0
$
otherwise. 
Clearly, $K(a) >0$ for all $a$ such that  $a> (x_0-x_m)/g_*$ and $r_\phi(t)$ is continuous on $\R_+$, with 
$r_\phi(0)=b(0) >0$. Then applying Corollary B.6 in \cite{ST}, we conclude that $b(t) >0$ for all large $t$. 
\qed

\section*{Acknowledgments}    \noindent  This research was supported by the projects FONDECYT 1231169 and AMSUD220002 (ANID, Chile). The authors are indebted to Pierre Magal and Quentin Griette for motivating discussions during the Winter School on Mathematical Modelling in Epidemiology and Medicine 2023, Valpara\'iso, Chile,  on the approaches used in  
\cite{HT,MM},  and  for calling authors'  attention to the works \cite{MZ,HS,GW}.

\vspace{0mm}

\section*{Data availability statement} 

 \noindent   Our manuscript has no associated data. 

\appendix

\end{document}